\documentclass[11pt]{amsart}
\usepackage{amsmath,amssymb,amsthm}
\usepackage{url}
\usepackage{cleveref}
\usepackage[margin=1.2in]{geometry}
\newtheorem{theorem}{Theorem}[section]
\newtheorem{lemma}[theorem]{Lemma}
\newtheorem{proposition}[theorem]{Proposition}
\newtheorem{corollary}[theorem]{Corollary}
\newtheorem*{conjecture}{Conjecture}
\theoremstyle{remark}
\newtheorem{remark}[theorem]{Remark}

\newcommand{\Z}{\mathbb{Z}}
\DeclareMathOperator{\dist}{dist}
\DeclareMathOperator{\Rea}{Re}
\DeclareMathOperator{\Ima}{Im}

\title{Integer values of $\tan(\arctan 1+\arctan 2+\cdots+\arctan n)$ are rare}

\author{Ken Ono}
\address{Axiom Math, 124 University Avenue, Palo Alto, CA 94301}
\email{ken@axiommath.ai}
\date{\today}


\begin{document}

\begin{abstract}
For $n\ge1$, we let $$x_n:=\tan\bigl(\sum_{k=1}^{n}\arctan k\bigr).$$
In 2008, Amdeberhan, Medina, and Moll conjectured that $x_n\notin\Z$
for every $n\ge5$. This was known for a set of positive integers of density
$\tfrac{120}{817}\approx0.1469$.  We prove that an integer value
$x_n=m$ satisfies $|m|\ge e^{(1/2+o(1))\,n\log n}$, which we use to deduce that
$$\#\{\,1\leq n\le N:x_n\in\Z\,\}=O(\log N).
$$
 In particular, the conjecture
holds for a density-one set of $n$. The results in this note were formalized in Lean/Mathlib
and produced autonomously by AxiomProver from natural-language statements.
\end{abstract}

\maketitle

\section{Introduction}\label{sec:intro}

In a study of arctangent sums, Boros and Moll \cite{BorosMoll2005}
considered the tangent of the running sum
$\arctan1+\cdots+\arctan n$,
\begin{equation}\label{eq:def-xn}
x_n \; :=\; \tan\Bigl(\sum_{k=1}^{n}\arctan k\Bigr).
\end{equation}
The addition formula for the tangent shows that $x_n$ is rational and
satisfies
\begin{equation}\label{eq:recurrence}
x_n \;=\; \frac{x_{n-1}+n}{1-n\,x_{n-1}},\qquad x_1=1.
\end{equation}
The sequence opens
\[
1,\ -3,\ 0,\ 4,\ -\tfrac{9}{19},\ \tfrac{105}{73},\
-\tfrac{308}{331},\ \tfrac{36}{43},\ \dots,
\]
and the integers cease abruptly. The first four terms are integers,
and apparently none are thereafter.  Amdeberhan, Medina, and Moll
\cite{AMM2008} proved $x_n\ne0$ for $n\ge4$ by computing $\nu_2(x_n)$,
and conjectured the following.

\begin{conjecture}[Amdeberhan--Medina--Moll {\cite{AMM2008}}]
If $n\ge5$, then $x_n$ is not an integer.
\end{conjecture}

The study of such arctangents goes back to Todd~\cite{Todd1949},
who considered integers \(m\) for which \(\arctan m\) can be written as a
\(\mathbb Z\)-linear combination of arctangents of smaller positive
integers.  Todd showed that the resulting notions of reducibility and
irreducibility are both infinite phenomena. There are infinitely many
integers admitting such reductions and infinitely many that do not.

In the present paper, we study a more rigid variant, suggested by
Amdeberhan, Medina, and Moll, in which the smaller arguments are forced
to be exactly \(1,2,\ldots,n\), each occurring with coefficient \(+1\).
Their conjecture asserts that, apart from the four known cases, this
rigid family produces no further integer values.

It turns out that the $x_n$ are governed by the arithmetic of Gaussian integers.
The relevant arithmetic enters through the expression
\[
        Z_n:=\prod_{k=1}^n(1+ik)=A_n+iB_n,\qquad A_n,B_n\in\mathbb Z.
\]
Then we have that
\[
        \arg Z_n=\arctan 1+\arctan 2+\cdots+\arctan n
        \pmod{\pi},
\]
and hence, whenever \(A_n\ne 0\), we have
\[
        x_n=\tan(\arg Z_n)=\frac{B_n}{A_n}.
\]
Moreover, we have
\[
        A_n^2+B_n^2=|Z_n|^2
        =\prod_{k=1}^n(1+k^2)=:\omega_n .
\]
Therefore, the question whether \(x_n\) is an integer is simultaneously an
analytic question about the position of the angle \(\arg Z_n\) relative
to the poles of the tangent function, and an arithmetic question about
the factorization of \(\omega_n\).

Several partial results toward the conjecture were known before the
present work.  Amdeberhan, Medina, and Moll proved a number of basic
structural facts, including that two consecutive values \(x_{n-1}\) and
\(x_n\) cannot both be integers~\cite[Thm.~4.1]{AMM2008}.  They also gave
criteria excluding integer values in certain arithmetic situations; for
example, if \(1+n^2\) is prime, then \(x_n\) is not an integer.  This
criterion is conditional in scope, since the infinitude of such \(n\) is
Landau's classical problem on primes of the form \(1+n^2\).  Their work
also relates integer values of \(x_n\) to the square structure of
\[
        \omega_n=\prod_{k=1}^n(1+k^2),
\]
while Cilleruelo proved that this product is a square only for
\(n=3\)~\cite{Cilleruelo2008}.  Thus this particular obstruction does not
settle the conjecture beyond the small exceptional case.

The strongest unconditional progress prior to this paper is due to
Moll~\cite{Moll2013}.  Using congruence properties of the sequence
\[
        f_n=(-1)^{n+1}\Re\prod_{k=0}^n(1+ik),
\]
he proved that \(x_n\) is not an integer for
\[
        n\equiv 5,13 \pmod {19}
        \qquad\text{and}\qquad
        n\equiv 8,34 \pmod {43}.
\]
These congruence classes have combined density
\[
        1-\frac{17}{19}\cdot\frac{41}{43}
        =\frac{120}{817}\approx 0.1469.
\]
Consequently, before the present work, the conjecture remained open for
about \(85\%\) of all positive integers \(n\).

We settle the conjecture for almost every $n$.  The key point is to use
this identity prime by prime.  For a prime $p$, let $\nu_p$ denote the
$p$-adic valuation, and define the \emph{squarefree kernel}
\begin{equation}\label{eq:def-K}
K_n\;=\!\!\prod_{\nu_p(\omega_n)\ \mathrm{odd}}\!\! p,
\end{equation}
the product of the primes dividing $\omega_n$ to an odd power.  If
$x_n=m\in\Z$, then the identity above gives
\[
        \omega_n=(1+m^2)A_n^2.
\]
The square factor $A_n^2$ contributes only even exponents to
$\omega_n$, and hence every prime occurring to odd order in
$\omega_n$ must divide $1+m^2$.  Equivalently, $K_n\mid 1+m^2$.
This yields our main theorem.

\begin{theorem}\label{thm:main} If $x_n=m$ is an integer, then the following hold.

\smallskip
\noindent
{\rm(1)} We have that $K_n\mid(1+m^2)$.

\smallskip
\noindent
{\rm(2)} If $K_n>1$, then
 $|m|\ge\sqrt{K_n-1}$ and
 \[
\log|m|\;\ge\;\Bigl(\tfrac12+o(1)\Bigr)\,n\log n .
\]
\end{theorem}

By Theorem~\ref{thm:main} (2), any integer value $x_n=m$ must satisfy
\[
        |m|\ge \sqrt{K_n-1}
        \qquad {\text {\rm and}}\qquad  \log |m|\ge \bigl(\tfrac12+o(1)\bigr)n\log n.
\]
Thus an integer value, if it exists, is far larger than any fixed power
of $n$.  We therefore isolate the indices at which $x_n$ is already
large on the elementary scale $n$:
\begin{equation}\label{eq:def-E}
E\;=\;\bigl\{\,n\ge5:\ |x_n|>\tfrac n2+1\,\bigr\}.
\end{equation}
By definition, if $n\notin E$, then $|x_n|\le n/2+1$.  For all
sufficiently large such $n$, this contradicts the lower bound
$|m|\ge\sqrt{K_n-1}$, and hence $x_n$ cannot be an integer.  It remains
only to count the exceptional set $E$; we prove below that
\[
        \#(E\cap[1,N])=O(\log N).
\]

\begin{corollary}\label{cor:density}
There is an effectively computable $n_0$ with $x_n\notin\Z$ for all
$n\ge n_0$, $n\notin E$.  Moreover, we have that
$$\#\bigl(E\cap[1,N]\bigr)=O(\log
N).$$
In particular, the conjecture holds on a set of density one.
\end{corollary}

The improvement over density $\tfrac{120}{817}$ is qualitative. The
unresolved set drops from positive density to logarithmic size.  The
kernel bound also explains the four integer values, each being the
least $m$ with $K_n\mid(1+m^2)$ at its index
(Section~\ref{sec:remarks}), the barrier becoming binding at $n=5$.

The proof is fairly elementary; it relies on standard facts such as Stirling's formula, Mertens' theorem, and Chebyshev's bound for the prime counting function.  Section~\ref{sec:barrier} establishes
Theorem~\ref{thm:main}, the growth of $K_n$ resting on two classical
estimates recalled there. Section~\ref{sec:exceptional} bounds $E$ and
proves Corollary~\ref{cor:density}. Section~\ref{sec:remarks} records
sharpness, two consequences, and the obstruction at the remaining
indices.  Finally, in Appendix~\ref{sec:AI}, we describe how AxiomProver, an AI
system under development, autonomously formalized and verified in Lean the
divisibility barrier of Theorem~\ref{thm:main}, the proximity bound behind
Lemma~\ref{lem:E}, and the density estimate of Corollary~\ref{cor:density}.

\section{The divisibility barrier}\label{sec:barrier}

We view $x_n$ through $\Z[i]$, and we let
\[
P_n=\prod_{k=1}^{n}(1+ik),\qquad A_n=\Rea P_n,\quad B_n=\Ima P_n .
\]
As $\arg(1+ik)=\arctan k$, we have $\arg P_n\equiv\sum_{k\le
n}\arctan k\pmod{2\pi}$. Since \eqref{eq:recurrence} never has a zero
denominator \cite[Cor.~2.7]{AMM2008}, $A_n\ne0$ and
\begin{equation}\label{eq:xn-AB}
x_n=\frac{B_n}{A_n}.
\end{equation}
Taking norms and using $|1+ik|^2=1+k^2$, we have
\begin{equation}\label{eq:norm}
A_n^2+B_n^2=|P_n|^2=\omega_n .
\end{equation}

The identity \eqref{eq:norm} lets us prove Theorem~\ref{thm:main},
which we restate here for convenience.

\begin{theorem}[$=$ Theorem~\ref{thm:main}]\label{thm:kernel}
If $x_n=m$ is an integer, then the following hold.

\smallskip
\noindent (1)  For every prime $p$ with $\nu_p(\omega_n)$ odd we have
$p\mid(1+m^2)$. As a consequence, we have  $K_n\mid(1+m^2)$.

\smallskip
\noindent
(2)  If $K_n>1$, then $|m|\ge\sqrt{K_n-1}$, and hence
$\log|m|\ge(\tfrac12+o(1))\,n\log n$.
\end{theorem}

\begin{proof}
By \eqref{eq:xn-AB}, $x_n=m$ gives $B_n=mA_n$, so \eqref{eq:norm}
becomes
\begin{equation}\label{eq:square-identity}
\omega_n=(1+m^2)\,A_n^2 ,
\end{equation}
which is \cite[Prop.~5.3]{AMM2008}.

\noindent\emph{Part (1).}  For each prime $p$, equation
\eqref{eq:square-identity} gives
\[
\nu_p(\omega_n)=\nu_p(1+m^2)+2\,\nu_p(A_n)\equiv\nu_p(1+m^2)\pmod2 ,
\]
so if $\nu_p(\omega_n)$ is odd then $\nu_p(1+m^2)\ge1$, that is,
$p\mid(1+m^2)$.  The primes dividing $\omega_n$ to an odd power are
distinct, so their product $K_n$ divides $1+m^2$ as well.

\noindent\emph{Part (2).}  Assume $K_n>1$.  Since $K_n\mid(1+m^2)$ by
part (1) and $1+m^2>0$, we have $1+m^2\ge K_n$, whence
$|m|\ge\sqrt{K_n-1}$.  Taking logarithms and applying
Lemma~\ref{lem:growth}, which gives $\log K_n=(1+o(1))\,n\log n$,
yields $\log|m|\ge\tfrac12\log(K_n-1)=(\tfrac12+o(1))\,n\log n$.
\end{proof}

The theorem is effective only if $K_n$ is large.  The first step is
that large primes cannot repeat in $\omega_n$.

\begin{lemma}\label{lem:rough}
If $p>2n$ is prime, then $\nu_p(\omega_n)\le1$.  Consequently
$R_n:=\prod_{p>2n}p^{\nu_p(\omega_n)}$ is squarefree and divides
$K_n$.
\end{lemma}

\begin{proof}
If $p\mid(1+j^2)$ and $p\mid(1+k^2)$ with $1\le j\le k\le n$, then
$j^2\equiv k^2\pmod p$, so $p\mid(k-j)$ or $p\mid(k+j)$; but
$0\le k-j<p$ and $0<k+j\le2n<p$, forcing $j=k$.  Thus $p$ divides at
most one factor of $\omega_n$, and only once as $p^2>4n^2>1+k^2$.  Each
such $p$ therefore occurs to the first power in $\omega_n$ and lies in
$K_n$.
\end{proof}

The growth of $K_n$ rests on two standard estimates, which we state
for reference.  The first is Stirling's formula in the form
\begin{equation}\label{eq:stirling}
\log n!=n\log n-n+O(\log n).
\end{equation}
The second is Mertens' theorem for the progression $1\bmod4$. Namely, there is
a constant $M$ with
\begin{equation}\label{eq:mertens}
\sum_{\substack{p\le x\\ p\equiv1\,(4)}}\frac{\log p}{p-1}
=\tfrac12\log x+M+o(1)\qquad(x\to\infty),
\end{equation}
the classical Mertens estimate distributed over the reduced residues
modulo $4$ \cite[Ch.~7]{Apostol1976}, with explicit error terms in
\cite{RamareRumely1996}.  We also use Chebyshev's bound
$\pi(x)=O(x/\log x)$ \cite[Thm.~4.6]{Apostol1976}, where $\pi(x)$
counts the primes up to $x$.

\begin{lemma}\label{lem:growth}
There is an absolute, effectively computable constant $C$ such that for all $n\geq 2$ we have
$$\log K_n\ge\log R_n\ge n\log n-Cn.
$$  In particular, we have that
$\log K_n=(1+o(1))\,n\log n$.
\end{lemma}

\begin{proof}
Since $R_n=\omega_n\big/\prod_{p\le2n}p^{\nu_p(\omega_n)}$, we have
\[
\log R_n=\log\omega_n-\sum_{p\le2n}\nu_p(\omega_n)\log p .
\]
For the first term, $\omega_n=(n!)^2\prod_{k\le n}(1+k^{-2})$ with
$1\le\prod_{k\ge1}(1+k^{-2})=\sinh\pi/\pi$, so by \eqref{eq:stirling},
\begin{equation}\label{eq:logomega}
\log\omega_n=2\log n!+O(1)=2n\log n-2n+O(\log n).
\end{equation}
For the second term, only $p=2$ and $p\equiv1\pmod4$ divide
$\omega_n$, since $1+k^2\equiv0\pmod p$ needs $-1$ a square modulo
$p$.  Here $\nu_2(\omega_n)=\lfloor(n+1)/2\rfloor$, as $\nu_2(1+k^2)$
is $1$ for odd $k$ and $0$ for even $k$.  For $p\equiv1\pmod4$ and
$i\ge1$, the congruence $k^2\equiv-1\pmod{p^i}$ has two solutions
modulo $p^i$, giving at most $2(n/p^i+1)$ admissible $k\le n$ and none
once $p^i>1+n^2$. Therefore, summing over $i$, we obtain
\[
\nu_p(\omega_n)\le\frac{2n}{p-1}+\frac{2\log(1+n^2)}{\log p}.
\]
Hence, by \eqref{eq:mertens} and Chebyshev's bound, we obtain
\[
\sum_{p\le2n}\nu_p(\omega_n)\log p
\le\frac{n+1}{2}\log2
+2n\!\!\sum_{\substack{p\le2n\\ p\equiv1\,(4)}}\!\frac{\log p}{p-1}
+2\log(1+n^2)\,\pi(2n)
= n\log n+O(n).
\]
Subtracting from \eqref{eq:logomega} gives $\log R_n\ge n\log n-Cn$.
\end{proof}

\begin{remark}\label{rem:constants}
Factoring the $1+k^2$ for $k\le3000$ gives $\log K_n\ge n\log n-4n$
throughout $5\le n\le3000$ (least margin $\approx18$, at $n=5$).
\end{remark}

\section{The exceptional set}\label{sec:exceptional}

By Theorem~\ref{thm:main} an integer value is impossible where $|x_n|$
is small, so only the large-value indices $E$ of \eqref{eq:def-E}
remain.  Writing $a_n=\sum_{k=1}^n\arctan\frac1k$, the identity
$$\arctan k=\tfrac\pi2-\arctan\tfrac1k$$
 turns \eqref{eq:def-xn} into
\begin{equation}\label{eq:theta}
x_n=\tan\theta_n,
\end{equation}
where $\theta_n=\tfrac{n\pi}{2}-a_n.$
Since $a_n$ grows only logarithmically, it rarely falls near
$\tfrac\pi2\Z$, which bounds $E$.

\begin{lemma}\label{lem:E} We have that
$\#\bigl(E\cap[1,N]\bigr)=O(\log N)$.
\end{lemma}

\begin{proof}
If $n\in E$ then $|x_n|>n/2$, so $\theta_n$ is within
$\arctan(2/n)<2/n$ of a pole; as $\tfrac{n\pi}2-\tfrac\pi2\in
\tfrac\pi2\Z$, \eqref{eq:theta} gives
\begin{equation}\label{eq:near-tau}
\dist\bigl(a_n,\tfrac\pi2\Z\bigr)<\tfrac2n .
\end{equation}
Because $a_n$ increases with $a_N\le\sum_{k\le N}1/k\le1+\log N$, the
points $\tau\in\tfrac\pi2\Z$ meeting \eqref{eq:near-tau} for some
$n\le N$ number at most $\tfrac2\pi(1+\log N)+1$.  Fix such a $\tau$
and let $n<n'$ in $E$ both satisfy \eqref{eq:near-tau} for it; then
\begin{equation}\label{eq:gap-small}
a_{n'}-a_n<\tfrac2n+\tfrac2{n'}\le\tfrac4n .
\end{equation}
From the alternating series, $\arctan t\ge t-t^3/3$ for $0<t\le1$, so
\begin{equation}\label{eq:arctan-lb}
\arctan\tfrac1{2n}\ge\tfrac1{2n}\bigl(1-\tfrac1{12n^2}\bigr)
\ge\tfrac2{5n}\qquad(n\ge1).
\end{equation}
If $n'>2n$, then $a_{n'}-a_n\ge\sum_{n<k\le2n}\arctan\tfrac1k\ge
n\cdot\tfrac2{5n}=\tfrac25>\tfrac4n$ for $n\ge12$, contradicting
\eqref{eq:gap-small}; so $n'\le2n$.  Then each of the $n'-n$ terms
$\arctan\tfrac1k$ ($n<k\le n'\le2n$) is $\ge\tfrac2{5n}$, and
\eqref{eq:gap-small} gives $(n'-n)\tfrac2{5n}<\tfrac4n$, i.e.
$n'-n<10$.  Thus, past $n=12$, each $\tau$ carries at most $10$
elements of $E$, and the count is $O(\log N)$.
\end{proof}

The indices in $E$ are also predictable.  As $\arctan\tfrac1k-\tfrac1k
=O(k^{-3})$ is summable, $a_n=\log n+c_0+O(1/n)$, where
\[
c_0=\gamma+\sum_{k=1}^{\infty}\Bigl(\arctan\tfrac1k-\tfrac1k\Bigr)
=0.301640\ldots
\]
and $\gamma=\lim_{n\to\infty}\bigl(\sum_{k\le n}\tfrac1k-\log n\bigr)$
is Euler's constant.

\begin{proposition}\label{prop:location}
Every $n\in E$ lies within $O(1)$ of some $n_j:=e^{\,j\pi/2-c_0}$ with
$j\ge2$.
\end{proposition}

\begin{remark}
To illustrate Proposition~\ref{prop:location}, we note that for  $j=2,\dots,7$ these values are $17.1,82.3,396.1,1905.2,9164.9,
44087.6$, and
\[
E\cap[5,6\cdot10^4]=\{15,17,\ 80,82,\ 395,397,\ 1904,\ 1906,\
9163,9165,\ 44086,44088\}.
\]
\end{remark}

\begin{proof}[Proof of Proposition~\ref{prop:location}]
By \eqref{eq:near-tau}, $|a_n-j\tfrac\pi2|<2/n$ for some $j$;
substituting $a_n=\log n+c_0+O(1/n)$ gives $\log
n=j\tfrac\pi2-c_0+O(1/n)$, i.e. $|n-n_j|=O(1)$.  The list is a direct
computation.
\end{proof}

\begin{remark}\label{rem:kappa}
Consecutive $n_j$ have ratio $e^{\pi/2}$, alternating between the even
and odd subsequences of $(x_n)$.  Along the even subsequence the ratio
is $e^{\pi}=23.14\ldots$, matching the empirical $\kappa_\infty
\approx23.1$ of \cite[\S7.3]{AMM2008}.  Thus \cite[Conj.~7.7]{AMM2008}
holds asymptotically with $\kappa_\infty=e^{\pi}$.
\end{remark}

\begin{proof}[Proof of Corollary~\ref{cor:density}]
With $C$ as in Lemma~\ref{lem:growth}, choose $n_0$ so that
$e^{(n\log n-Cn)/2}-1>n/2+1$ for $n\ge n_0$.  If $n\ge n_0$, $n\notin
E$, and $x_n=m\in\Z$, then Theorem~\ref{thm:main} and
Lemma~\ref{lem:growth} give
$\tfrac n2+1\ge|m|\ge\sqrt{K_n-1}\ge\sqrt{e^{n\log n-Cn}-1}>\tfrac
n2+1$, a contradiction.  The count is Lemma~\ref{lem:E}.  \end{proof}

\section{Sharpness, consequences, and the exceptional indices}
\label{sec:remarks}

\subsection*{Sharpness}
The divisibility of Theorem~\ref{thm:main} is attained at every known
integer value.  Table~\ref{tab:sharp} lists $K_n$, the least $|m|$
with $K_n\mid(1+m^2)$, and $x_n$.

\begin{table}[ht]
\centering
\begin{tabular}{|c|c|c|c|}
\hline
$n$ & $K_n$ & $\min\{|m|:K_n\mid(1+m^2)\}$ & $x_n$\\
\hline
$1$ & $2$ & $1$ & $1$\\ \hline
$2$ & $10$ & $3$ & $-3$\\ \hline
$3$ & $1$ & $0$ & $0$\\ \hline
$4$ & $17$ & $4$ & $4$\\ \hline
$5$ & $442$ & $21$ & $-9/19$\\ \hline
$6$ & $16354$ & $931$ & $105/73$\\ \hline
\end{tabular}
\smallskip
\caption{For $n\le4$ the value $x_n$ is the minimal solution of
$K_n\mid(1+m^2)$; from $n=5$ on, the least admissible $|m|$ exceeds every
attainable value.}
\label{tab:sharp}
\end{table}

Each of the four integer values is thus the least $m$ with
$K_n\mid(1+m^2)$.  The barrier first binds at $n=5$, where
$K_5=442=1+21^2$: the minimal level is $m=21$, and
$(1+21^2)\omega_5=4420^2$---the example of J.~McLaughlin
\cite[p.~1825]{AMM2008}---since $\omega_5=442\cdot10^2$.  What fails at
$n=5$ is the angle, not the arithmetic: $x_5=-\tfrac9{19}$.

\subsection*{Two consequences}
The kernel bound weakens the primality hypothesis of
\cite[Thm.~4.5]{AMM2008} to a single large prime factor of $\omega_n$.

\begin{corollary}\label{cor:onelarge}
Let $n\notin E$.  If $\omega_n$ has a prime factor
$p\ge(\lfloor n/2\rfloor+2)^2$, then $x_n\notin\Z$.
\end{corollary}

\begin{proof}
Such $p$ exceeds $2n$, so $p\mid K_n$ by Lemma~\ref{lem:rough}; an
integer $m$ would give $$1+m^2\ge p>1+(n/2+1)^2\ge1+m^2.$$
\end{proof}

It also gives a Pell-type reformulation.

\begin{proposition}\label{prop:pell}
If $x_n=m\in\Z$, then $m^2-K_nu^2=-1$ for some $u\in\Z_{\ge1}$.
\end{proposition}

\begin{proof}
Write $\omega_n=K_ns^2$.  By \eqref{eq:square-identity},
$(1+m^2)/K_n=(s/A_n)^2$ is a rational square that is a positive
integer by Theorem~\ref{thm:main}, hence a perfect square $u^2$.
\end{proof}

The case $m=0$ of \eqref{eq:square-identity} reads $\omega_n=A_n^2$,
so Cilleruelo's theorem \cite{Cilleruelo2008} recovers $x_n\ne0$ for
$n\ne3$ \cite[Thm.~2.6]{AMM2008}. The conjecture is the union over
levels $m$ of statements whose $m=0$ case already required a proof.

\subsection*{The exceptional indices}
The remaining indices near $n_j=e^{j\pi/2-c_0}$ appear genuinely hard.
An integer value there would place $\theta_n$ within $\asymp
e^{-h}$ of a pole, where $h$ is the height of $(1+im)/(1-im)$.  The
associated linear form in two logarithms would then have to be smaller
than the trivial scale, beyond the reach of even conjecturally sharp
lower bounds; the only near-equalities of this kind are the Machin-type
relations at $n\le4$.  The arithmetic side is not easier.
By Proposition~\ref{prop:pell}, integrality needs $K_n$ of the form
$(1+m^2)/u^2$, which does occur ($K_4=17=1+4^2$ is realized,
$K_5=442=1+21^2$ is not).  A proof must couple the square class of
$\omega_n$ to the angle.

Finally, the two methods are complementary. Of the ten exceptional
indices below $2\cdot10^4$, $9163\equiv5\!\pmod{19}$ falls under
\cite[Thm.~1.6]{Moll2013} and $395\equiv8\!\pmod{43}$ under
\cite[Note~7.3]{Moll2013}, so the congruence certificates of
\cite{Moll2013} reach into exactly the set where the present method fails to give information.

\subsection*{Acknowledgements}
The author thanks Tewodros Amdeberhan for bringing this problem to
his attention, and thanks Jujian Zhang for assembling the GitHub repository for the relevant AxiomProver files and artifacts.

\appendix
\section{AxiomProver's autonomous Lean formalization}\label{sec:AI}

AxiomProver is an AI system for mathematical research via formal proof that is
currently under development.  As an early test case, we treated the results of
this paper as an end-to-end formalization target: the formal statements and
their proofs were produced by AxiomProver from a natural-language statement of
the problem, and were then verified by the Lean proof assistant.  Using that
formal development as a reference, the human author wrote the exposition in the
main text for a mathematical audience.  Readers not interested in automated
theorem proving may skip this appendix.

\subsection*{Scope of the formalization}
The formalization concerns the three main results of this paper, stated in Lean
for the Gaussian integers $P_n=\prod_{k=1}^{n}(1+ik)$, the rational sequence
$x_n=\Ima(P_n)/\Rea(P_n)$, and the arctangent angle sum
$a_n=\sum_{k=1}^{n}\arctan(1/k)$.  Writing $A_n=\Rea(P_n)$, $B_n=\Ima(P_n)$,
$\omega_n=A_n^2+B_n^2=\prod_{k=1}^{n}(1+k^2)$, and letting $K_n$ be the
squarefree kernel of $\omega_n$, the formalized statements are the following.
\begin{itemize}
\item \texttt{thm\_main} --- If $x_n=m$ is an integer (with $A_n\ne0$), then
  $K_n\mid(1+m^2)$, and if $K_n>1$ then $|m|\ge\sqrt{K_n-1}$.  This is
  Theorem~\ref{thm:main}(1) together with the divisibility bound of
  Theorem~\ref{thm:main}(2); the asymptotic form
  $\log|m|\ge(\tfrac12+o(1))\,n\log n$, which relies on the growth estimate of
  Lemma~\ref{lem:growth}, is part of the human exposition and was not included
  in the formalized statement.
\item \texttt{lem\_proximity} --- For every $n$ in the exceptional set
  $E=\{\,n\ge5:|x_n|>n/2+1\,\}$, there is an integer $j$ with
  $|a_n-j\tfrac\pi2|<\tfrac2n$.  This is the proximity bound
  \eqref{eq:near-tau} underlying Lemma~\ref{lem:E}.
\item \texttt{cor\_density} --- $\#\bigl(E\cap[1,N]\bigr)=O(\log N)$, which is
  Lemma~\ref{lem:E}, and hence the conjecture holds on a set of density one.
\end{itemize}

\subsection*{Process}
The formal proofs were developed and verified using Lean \textbf{4.28.0};
compatibility with earlier or later versions is not guaranteed, owing to the
evolving nature of the Lean~4 compiler and its core libraries.  The proofs are
fully sorry-free: they add no axioms and leave no stray \texttt{sorry}.  All
relevant files are posted in the repository
\begin{center}
  \url{https://github.com/AxiomMath/TanArctan}
\end{center}
The input given to AxiomProver consisted of
\begin{itemize}
\item \texttt{input/main\_engine.tex}, a self-contained natural-language
  statement of the definitions and of the three target results; and
\item \texttt{input/task.md}, which instructs the system to read
  \texttt{main\_engine.tex} and to formalize and prove the theorem, lemma, and
  corollary with a fully sorry-free proof.
\end{itemize}
From these inputs, AxiomProver autonomously produced
\begin{itemize}
\item \texttt{problem.lean}, a Lean formalization of the problem statement; and
\item \texttt{solution.lean}, a complete Lean formalization of the proof.
\end{itemize}
After AxiomProver generated the solution, the human author wrote this paper for
human readers.  A research paper is a narrative designed to communicate ideas to
people, whereas a Lean file is written to satisfy a proof-checking kernel. At
first glance, therefore, the formal proofs do not resemble the narrative
presented here.

\section*{Declaration of generative AI and AI-assisted technologies in the manuscript preparation process}
As described in the preceding appendix, AxiomProver was used to produce and
formally verify, in Lean, the proofs of the results of this paper.  The
exposition was subsequently written by the human author.

\end{document}